\documentclass[12pt]{article}
\usepackage{prettyref}
\usepackage{amsmath}
\usepackage{amsthm}
\usepackage{amssymb}
\usepackage{esint}
\usepackage{hyperref}
\usepackage{cite}
\usepackage{appendix}
\usepackage[normalem]{ulem}

\makeatletter
\numberwithin{equation}{section}
\numberwithin{figure}{section}
\theoremstyle{plain}
\newtheorem{thm}{\protect\theoremname}[section]
  \theoremstyle{definition}
  
  \theoremstyle{remark}
  \newtheorem{rem}[thm]{\protect\remarkname}
  \theoremstyle{plain}
  \newtheorem{cor}[thm]{\protect\corollaryname}
      \newtheorem*{fact*}{Fact}
  \theoremstyle{plain}
  \newtheorem{lem}[thm]{\protect\lemmaname}
  \theoremstyle{definition}

  \theoremstyle{plain}
  
  \theoremstyle{remark}
  \newtheorem*{rem*}{\protect\remarkname}

\usepackage{mathrsfs,amsthm,amssymb,bbm,fullpage}
\usepackage{bbm,mathrsfs}

\usepackage{xcolor}
\usepackage[normalem]{ulem}

\hypersetup{colorlinks=true,pdfborder=000,  citecolor  = blue, citebordercolor= {magenta}}
\hypersetup{linkcolor=purple}
\makeatletter
\let\reftagform@=\tagform@
\def\tagform@#1{\maketag@@@{(\ignorespaces\textcolor{purple}{#1}\unskip\@@italiccorr)}}
\renewcommand{\eqref}[1]{\textup{\reftagform@{\ref{#1}}}}
\makeatother

\makeatletter

\DeclareUrlCommand\ULurl@@{%
  \def\UrlLeft{\uline\bgroup}%
  \def\UrlRight{\egroup}}
\def\ULurl@#1{\hyper@linkurl{\ULurl@@{#1}}{#1}}
\DeclareRobustCommand*\ULurl{\hyper@normalise\ULurl@}
\makeatother

\newcommand{\cP}{\mathcal{P}}
\newcommand{\cN}{\mathcal{N}}

\newcommand{\cE}{\mathcal{E}}

\newcommand{\R}{\mathbb{R}}

\newcommand{\N}{\mathbb{N}}

\newcommand{\prob}{\mathbb{P}}
\newcommand{\E}{\mathbb{E}}

\newcommand{\eps}{\epsilon}

\newcommand{\indicator}[1]{\mathbbm{1}\left\{{#1}\right\}}

\newcommand{\eqdist}{\stackrel{(d)}{=}}
\newcommand{\convdist}{\stackrel{(d)}{\longrightarrow}}
\newcommand{\tensor}{\otimes}

\newcommand{\abs}[1]{\left\lvert#1\right\rvert}
\newcommand{\norm}[1]{\lvert\lvert#1\rvert\rvert}

\newcommand{\gibbs}[1]{\left\langle #1\right\rangle}

\newrefformat{cor}{Corollary \ref{#1}}
\newrefformat{sub}{Section \ref{#1}}
\newrefformat{exam}{Example \ref{#1}}
\newrefformat{reg}{Regime \ref{#1}}
\newrefformat{reg}{Regime \ref{#1}}
\newrefformat{ass}{Assumption \ref{#1}}
\newrefformat{rem}{Remark \ref{#1}}
\newrefformat{app}{Appendix \ref{#1}}

\makeatother

  \providecommand{\corollaryname}{Corollary}
  \providecommand{\definitionname}{Definition}
  \providecommand{\examplename}{Example}
  \providecommand{\lemmaname}{Lemma}
  \providecommand{\propositionname}{Proposition}
  \providecommand{\remarkname}{Remark}
\providecommand{\theoremname}{Theorem}

\renewcommand{\limsup}{\varlimsup}

\begin{document}

\title{Thouless-Anderson-Palmer equations  for generic $p$-spin glasses}

\author{Antonio Auffinger \thanks{Department of Mathematics, Northwestern University, auffing@math.northwestern.edu} \\
\small{Northwestern University}\and Aukosh Jagannath \thanks{Department of Mathematics, Harvard University, aukosh@math.harvard.edu}\\
\small{Harvard University}}

\date{\today}
\maketitle
\begin{abstract}
We study the Thouless-Anderson-Palmer (TAP) equations for spin glasses on the hypercube.
First, using a random, approximately ultrametric decomposition of the hypercube, we decompose the
Gibbs measure, $\gibbs{\cdot}_N$, into a mixture of conditional laws, $\gibbs{\cdot}_{\alpha,N}$.
We show that the TAP equations hold for the spin at any site with respect to $\gibbs{\cdot}_{\alpha,N}$
simultaneously for all $\alpha$. This result holds for generic models
provided
that the Parisi measure of the model has a jump at the top of its support.  
\end{abstract}

\section{Introduction}

The Thouless-Anderson-Palmer (TAP) equations were introduced by Thouless,
Anderson, and Palmer \cite{TAP77} as the mean field equations for
the Sherrington-Kirkpatrick (SK) model of spin glasses. These equations
can be stated informally as follows. For each $\sigma\in\Sigma_{N}=\{-1,1\}^{N}$,
let 
\[
H_{N}(\sigma)=\frac{1}{\sqrt{N}}\sum_{i,j=1}^{N}g_{ij}\sigma_{i}\sigma_{j}
\]
be the Hamiltonian for the SK model. Here $g_{ij}$ are i.i.d. standard
Gaussian random variables for $1\leq i \leq j \leq N$ and $g_{ij}=g_{ji}$. Let 
\[
\mu_{N}(\{\sigma\})=\frac{e^{-\beta H_{N}(\sigma)+h\sum_{i=1}^{N}\sigma_{i}}}{Z_{N}}
\]
be the Gibbs measure of this system at inverse temperature, $\beta$,
and external field, $h$. Here $\beta\text{ and }h$ are non-negative
real numbers and $Z_{N}$ is chosen such that $\mu_{N}$ is a probability
measure on $\Sigma_{N}$. We denote integration of a quantity,
say $\sigma_{i}$, against $\mu_N$ as $\left\langle \sigma_{i}\right\rangle $. The
TAP equations state that in the limit that $N\to\infty$, we have
that 
\begin{equation}
\left\langle \sigma_{i}\right\rangle _{\alpha}\approx\tanh\left(h+\left\langle 
\frac{1}{\sqrt{N}}\beta\sum_{j}(g_{ij}+g_{ji})\sigma_{j}\right\rangle _{\alpha}-\beta^{2}(1-q_{*})\left\langle \sigma_{i}\right\rangle _{\alpha}\right),\label{eq:tap-informal}
\end{equation}
 for some $q_{*}\in[0,1]$ and for some random measure for which integration is denoted by $\langle \cdot \rangle_{\alpha}$. 

There have been two approaches to proving
the TAP equations rigorously. 
The first approach is to take $\left\langle \cdot\right\rangle _{\alpha}$
as integration with respect to the Gibbs measure. This has been done
by Talagrand \cite{TalBK03} and Chatterjee \cite{Chat10} at sufficiently
high temperature for the SK model where they establish \prettyref{eq:tap-informal}
under this interpretation. A second approach, introduced by Bolthausen
\cite{Bolt14}, is to interpret $\left\langle \sigma_{i}\right\rangle _{\alpha}$
as a vector in high dimensions, and to understand \prettyref{eq:tap-informal}
through a fixed point iteration scheme. There he showed that this
iteration converges to a unique solution of \prettyref{eq:tap-informal}
in the entire predicted high temperature regime. 
At low temperature, as far as we know,  there is no rigorous proof
 of \prettyref{eq:tap-informal}. 
 In this regime, it is expected that there are many distinct measures,
 $\mu_{\alpha,N}$, called ``pure states'', whose convex combination
 is $\mu_N$  and each of which satisfies \prettyref{eq:tap-informal}.
 
The first goal of this paper is to study \prettyref{eq:tap-informal} for  generic mixed $p$-spin glasses without an assumption on the
temperature. These models are defined as follows. Consider the mixed $p$-spin
glass Hamiltonian, $H_{N}(\sigma)$, which is the centered Gaussian
process on $\Sigma_{N}=\{-1,1\}^{N}$ with covariance 
\[
\E H_{N}(\sigma^{1})H_{N}(\sigma^{2})=N\xi(R_{12}),
\]
where $R_{12}=\frac{1}{N}\sum\sigma_{i}^{1}\sigma_{i}^{2}$ is called
the overlap and  $\xi(t)=\sum_{p\geq 2} \beta_{p} t^{p}$ is called the model. We let $\mu_{N}$ denote the corresponding Gibbs measure and $\langle \cdot \rangle$ expectation under products of $\mu_{N}$. 
The SK model corresponds to $\xi(t)=\beta_{2}t^{2}$. A mixed $p$-spin glass model is called \emph{generic} if the set $\{t^{p}:\beta_{p}>0\}$ is total in $(C([-1,1]),\sup\abs{\cdot})$. 

Denote by $\zeta_{N}$ the distribution of the overlap under the measure
 $\mathbb E \mu_{N}^{\otimes 2}$, that is,
 $$\zeta_{N}(A) = \mathbb E \langle \mathbf 1(R_{12} \in A) \rangle$$ 
 for any measurable $A\subset [-1,1]$. It is known that $\zeta_{N}$ converges to $\zeta$, where $\zeta$
is the unique minimizer of the Parisi formula \cite{AuffChenSC15,PanchSKBook}. It is also known that generic models  satisfy the Ghirlanda-Guerra identities in the limit \cite{PanchGhir10,PanchSKBook}. As a result, their asymptotic Gibbs measures \cite{ArgAiz09} are known to have ultrametric support by Panchenko's ultrametricity theorem \cite{PanchUlt13}. We assume that $\zeta$ has a jump at the top of its support. That is, if $q_{*}:=\sup \text{supp} (\zeta)$, we assume that
\begin{equation}\label{eq:assumptiontop}
\zeta(\{q_{*}\})>0.
\end{equation} 
This assumption is expected to hold in a wide range of models at all temperatures. For more on this see \prettyref{rem:jump}.

This  ultrametric structure is the starting point for our study of the analogue of \prettyref{eq:tap-informal} for generic models.
It was shown in \cite{Jag14} that, as a consequence of Panchenko's ultrametricity theorem, $\Sigma_N$ can be 
decomposed as the disjoint union of a collection of clusters, $\{C_{\alpha,N}\}_{\alpha\in\N}$, which satisfy
certain ultrametric-type properties. Heuristically, these clusters are essentially balls of radius $q_*$. 
Within a cluster, the points are at overlap roughly $q_*$,  between clusters
the points have overlap less than $q_*-o_N(1)$ with high $\mu_N$ probability.
We recall the precise definition of these sets in \prettyref{app:AU}. 
A similar decomposition was obtained by Talagrand in \cite{Tal09}.
 
 For each of these clusters, $C_{\alpha,N}$, we define 
\begin{equation}\label{eq:1}
\mu_{\alpha,N}(\cdot):=\mu_{N}(\cdot|C_{\alpha,N}).
\end{equation}
That is, $\mu_{\alpha,N}$ is the  Gibbs measure conditioned on the set $C_{\alpha,N}$, with the convention that if $C_{\alpha,N}=\emptyset$, then $\mu_{\alpha,N}=\delta_{(1,\ldots,1)}$.
This yields a decomposition of the Gibbs measure $\mu_{N}$ as
\begin{equation}
\mu_{N}(\cdot) = \sum_{\alpha} \mu_{\alpha,N}(\cdot) \mu_{N}(C_{\alpha,N})+o_N(1).
\end{equation} 
Here, $o_{N}(1)$ means that $\mu_{N}((\cup_{\alpha}C_{\alpha,N})^{c})$ goes to zero in probability as $N$ goes to infinity. The sets $C_{\alpha,N}$ are also ordered with respect to their Gibbs masses, that is,
\[
\mu_{N}(C_{1,N}) \geq  \mu_{N}(C_{2,N}) \geq \mu_{N}(C_{3,N}) \geq \ldots
\]
Integration with respect to the conditional measure $\mu_{\alpha,N}$ will be denoted by $\langle \cdot \rangle_{\alpha,N}$. 

We now state our main theorem, which is the equivalent of \eqref{eq:tap-informal}
for generic models.
For $\sigma \in \Sigma_{N}$, let 
\begin{equation}\label{spinlocalfield}
\quad y_{N}(\sigma)= \sum_{p\geq 2} \frac{\beta_{p}}{N^{\frac{p-1}{2}}} \sum_{2\leq i_{2},\ldots,i_{p}\leq N} J_{i_{2}\cdots i_{p}}\sigma_{i_{2}}\cdots\sigma_{i_{p}},
\end{equation}
with $J_{i_{2}\ldots i_{p}}= g_{1i_{2}\ldots i_{p}}+g_{i_{2}1\ldots i_{p}} + \ldots + g_{i_{2}\ldots i_{p}1}$, where $g_{i_{1}i_{2}\ldots i_{p}}$, $1\leq i_{1},\ldots,i_{p}\leq N$ are i.i.d. standard Gaussian random variables. 
We call $\sigma_1$ the spin of the first particle and $y_{N}$ the local field on the first particle.
Note that $y_{N}$ is a centered Gaussian process on $\Sigma_{N-1}$ with covariance given by 
\begin{equation}\label{eq:yn-def}
\mathbb E y_{N}(\sigma^{1})y_{N}(\sigma^{2}) = \xi'(N^{-1} (\sigma^{1},\sigma^{2})).
\end{equation} 
For more on $y_{N}$ see \prettyref{lem:decomposition-lemma}. We also note here that the choice of the 
first spin as opposed to any fixed $i$ will be irrelevant by site symmetry.

Our main result is that the TAP equation for a spin holds for the measures $\langle \cdot \rangle_{\alpha,N}$. 

\begin{thm}
\label{thm:TAP-finite-N} Assume that $\zeta(q_{*})>0$. We have that 

\begin{equation}\label{eq:TAPmixed}
\left( \langle \sigma_1  \rangle_{\alpha,N}-\tanh\left[\left\langle y_{N}\right\rangle _{\alpha,N}+h-(\xi'(1)-\xi'(q_{*})) \langle \sigma_1 \rangle_{\alpha,N}\right] \right)_{\alpha\in \N}\to0
\end{equation}
in distribution.
\end{thm}
The proof of Theorem \ref{thm:TAP-finite-N} has several steps and along the way we pick up results that 
are of independent interest. We will outline the proof of Theorem \ref{thm:TAP-finite-N} in the next section. 
We conclude this section with the following remarks.

\begin{rem}
At high temperature and with $h=0$,  the Parisi measure $\zeta = \delta_{0}$, and the decomposition $C_{\alpha,N}$ is given by $C_{1,N}=\Sigma_{N}$, $C_{\alpha,N}=\emptyset$, $\alpha>1$. The conditional measure $\mu_{1,N}$ is now identical to the Gibbs measure $\mu_{N}$ and one recovers the result of Talagrand \cite{TalBK03} for  a single spin. 
\end{rem}

\begin{rem}Theorem \ref{thm:TAP-finite-N} establishes the TAP equations for a single spin. The TAP equations are also predicted to hold for all spins $\sigma_{1}, \ldots, \sigma_{N}$ simultaneously.  
\end{rem}

\begin{rem}\label{rem:jump}
The assumption that the Parisi measure has a jump at the top of its support, $\zeta(q_{*})>0$, is believed to be true for a large collection of (if not all) generic models at all temperatures. Results in this direction were obtained by Auffinger-Chen (see Theorem 4 in \cite{AuffChenPM15}). If there is no jump at the top of the support, then it is unclear the extent to which
a true pure state decomposition will hold in such systems \cite{PanchHEpure15}. 
In a follow up paper \cite{AA2018}, we will show that at infinite particle number, \prettyref{eq:tap-informal}
holds without this assumption. In fact, we will show a multiscale generalization of these equations.
\end{rem}

\begin{rem}
Since the statement of Theorem \ref{thm:TAP-finite-N} depends on the construction of the measures $\langle \cdot \rangle_{\alpha,N}$, one may wonder what would happen if one takes a different decomposition. In Section \ref{sec:Nesting} we show that the decomposition \eqref{eq:1} is essentially unique in the following sense. Any other collection of subsets $X_{\alpha,N}$ that satisfy
the same properties as $C_{\alpha,N}$ must also satisfy $\mu_{N} (X_{\alpha,N} \Delta C_{\alpha,N}) \to 0$.  
\end{rem}

\subsection{Outline of the proof of Theorem \ref{thm:TAP-finite-N}}\label{sec:outline}
\prettyref{thm:TAP-finite-N} relates the quantities
\[
\gibbs{\sigma_1}_{\alpha,N} =\frac{1}{\mu_N( C_{\alpha,N})} \int_{C_{\alpha,N}} \sigma_1 d\mu_N 
\qquad \text{and}\qquad
\gibbs{y_N}_{\alpha,N}=\frac{1}{\mu_N(C_{\alpha,N})}\int_{C_{\alpha,N}} y_N(\sigma)d\mu_N.
\]
Put differently, we are interested in the relation between $\sigma_1$ and $y$ within a cluster, $C_\alpha$.
Heuristically, for large $N$ there is little difference between a fixed coordinate and a ``cavity coordinate''.
By a cavity coordinate, we mean that we study the law of $(s_{\alpha,N},y_{\alpha,N})$
which are distributed like $(\epsilon,y_{N}(\sigma))$ drawn from the tilted measure on $\Sigma_{N+1}$,
\[
d\mu_{N}^{\intercal}(\eps,\sigma)=\frac{e^{\epsilon y_{N}(\sigma)}d\eps\,d\mu_{N}(\sigma)}{\int2\cosh(y_{N}(\sigma))d\mu_{N}}
\]
conditioned on the event $\{\sigma\in C_{\alpha,N}\}$.  Call this conditional
measure $\mu_{\alpha,N}^{\intercal}$. 
Here, we assume that $y_{N}$ is independent
of $\mu_{N}$ and satisfies 
\[
\E y_{N}(\sigma^{1})y_{N}(\sigma^{2})=\xi'(R_{12})+o_{N}(1).
\]
As a result, to study convergence of $(s_{\alpha,N},y_{\alpha,N})$ for a fixed $\alpha$, it suffices
to study convergence of statistics of the form 
\[
\E\prod_{i}\int\phi_{i}(\epsilon,y_{N})d\mu_{\alpha,N}^{\intercal}
\]
for any finite family of reasonable $\phi_{i}$.
These statistics, as we will find, are continuous functionals of the
law of the overlap array of i.i.d. draws from $\mu_{\alpha,N}$.
The $\mu_{\alpha,N}$ are asymptotically
\emph{replica symmetric}, that is, their overlap array converges
to the matrix which is 1 on the diagonal and $q_{*}=\sup supp\{\zeta\}$
on the off diagonal. This implies that the law of $(s_{\alpha,N},y_{\alpha,N})$
converges to the law of a stochastic process, $(s,y)$,
which can be described as follows: let $h_\alpha$ be a centered gaussian with 
variance $\xi'(q_*)$. Then $(s,y)$ are the random variables with conditional density
\begin{equation}\label{eq:density-of-sy}
p(s,y;h_\alpha)\propto e^{-\frac{(y-h_\alpha)^2}{2(\xi'(1)-\xi'(q_*))}}e^{sy},
\end{equation}
with respect to the product of the counting measure on $\Sigma_1$ and Lebesgue measure on $\R$. 
It is an elementary calculation to show that this satisfies the TAP equation,
\begin{equation}\label{eq:tap-pure-state}
\gibbs{s}_\alpha = \tanh(\gibbs{y}_\alpha - (\xi'(1)-\xi'(q_*))\gibbs{s}_\alpha),
\end{equation}
conditionally on $h_\alpha$.
Indeed, once making this reduction, this is similar in spirit to the high temperature setting as in \cite{Chat10}. 
(This is stated and proved in a slightly more general setting in \cite{Chat10}.) 
This step is shown in \prettyref{sec:cavity}.

The final question is then: ``to what extent can we treat a fixed coordinate as a cavity coordinate?''. 
The answer comes by first showing that the collection $C_{\alpha,N} \times \{\pm1\}$ preserves most of the ultrametric properties after a (random) reshuffling. This is done in Sections \ref{sec:stability} and \ref{sec:Nesting}. We then use the replica symmetric structure of the conditional measures to deal with the dependence of $y_{N}$ on both the clusters and the Gibbs measure. This ends the proof of the theorem in Section \ref{sec:tap-fixed-coordinate}.  

\subsection*{Acknowledgements}

We thank an anonymous referee for a careful reading of this manuscript which led to numerous
helpful comments and suggestions which greatly improved the presentation
of this paper. 
The authors thank Dmitry Panchenko for numerous comments, suggestions and several discussions on a first version of this project, which dramatically improved the results of this paper. 
We also thank G\'erard Ben Arous and Ian Tobasco for several fruitful discussions. 
A.A. would also thank Louis-Pierre Arguin, Wei-Kuo Chen and
Nicola Kistler for helpful and broad discussions about TAP. A.J. thanks
the Northwestern University for their hospitality. This research was
conducted while A.A. was supported by NSF DMS-1597864 and A.J. was
supported by NSF OISE-1604232. 

\section{Convergence of Spins and Local fields for a Cavity Coordinate}
\label{sec:cavity}

In this section, we study the joint law of a spin and the local field on that spin 
for a cavity coordinate. 
As a consequence of this, we  find that \eqref{eq:TAPmixed}  holds for a cavity coordinate.

\textbf{Note:} In the remainder of this paper we take $h=0$. This does not change the arguments,
however it simplifies the notation.

Let $(H'(\sigma))$ be a centered Gaussian process on $\Sigma_{N}$ with covariance
\begin{equation}
\E H'(\sigma^{1})\cdot H'(\sigma^{2})=N\xi(R_{12})+o_{N}(1)\label{eq:ham-for-spin-section}
\end{equation}
where by the term $o_{N}(1)$, we mean a function of the overlap that vanishes uniformly
as $N$ tends to infinity. Let $\nu_N$ denote the Gibbs measure
on $\Sigma_{N}$ corresponding to $H'$.  Let $(y(\sigma))$ be a centered
Gaussian process on $\Sigma_N$ that is independent of $H'$ and satisfies
\begin{equation}\label{eq:y-defn-indept}
 \E y(\sigma^1)y(\sigma^2) = \xi'(R_{12})+o_N(1) 
\end{equation}
where again the $o_N(1)$ term is a function of the overlap.

Corresponding to $y$, we define a random tilt of $\nu_N$, which we denote by 
$\nu_N^{\intercal}$, as  the measure 
\begin{equation}\label{eq:nu-tilt}
\nu_N^\intercal = T(\sigma)d\nu_N
\end{equation}
where $T$ is given by
\begin{equation}\label{eq:rad-nyk}
T(\sigma)=\exp\left( \log( \cosh(y(\sigma))) - \log (\int_{\Sigma_{N}}  \cosh(y(\sigma))d\nu_N)\right).
\end{equation}
Observe that since $\cosh(x)\geq 1$, these measures are mutually absolutely continuous.

Assume that for $H'$ , the limiting overlap distribution satisfies $\zeta(q_{*})>0$. 
As $\xi$ is generic, there is a collection of sets, $\{X_{\alpha,N}\}\subset\Sigma_{N}$,
that satisfies items 1.-5. of Theorem~\ref{thm:pure-state-AU}, with respect to the measure
$\nu_N$.  We drop the $N$ dependence in the notation of $X_{\alpha,N}$ and write $X_{\alpha}$. For each $\alpha\in \mathbb N$, we define the measure
\[
\nu_{\alpha,N} = \nu_N(\cdot\vert X_\alpha),
\]
when $X_\alpha$ is non-null, and on the event that it is null, let this be $\delta_{(1,\ldots,1)}$. 
Finally we let $\nu_{\alpha,N}^\intercal $ be the measure on $\{-1,1\}\times\Sigma_N$ such that for $\phi$ 
continuous and bounded,
\begin{equation}\label{eq:nu-alpha-T}
\int \phi(s,\sigma) d\nu^\intercal_{\alpha,N} 
= \frac{\int_{X_\alpha}\int_{\Sigma_1} \phi(s,\sigma) e^{sy(\sigma)}ds d\nu_N(\sigma)}{\int_{X_\alpha} 2\cosh(y(\sigma))d\nu_N}.
\end{equation}
For the purposes of this section, let $\left\langle \cdot\right\rangle _{\alpha,N}$
denote integration with respect to $\nu_{\alpha,N}$, and $\gibbs{\cdot}_{\alpha,N}^\intercal$ to 
denote integration with respect to $\nu_{\alpha,N}^\intercal$.

Let $((s^i_{\alpha,N},\sigma^i_{\alpha,N}))_{i\geq 1}$ then be i.i.d. draws from $\nu^\intercal_{\alpha,N}$,
and let $y_{\alpha,N}^i = y(\sigma^i_{\alpha,N})$. 
The goal of this section is to study the convergence of the joint
law of $((s^i_{\alpha,N},y^i_{\alpha,N}))_{i\geq 1}$. In particular, let $h_\alpha\sim\cN(0,\xi'(q_*))$, and let 
$\nu_\alpha$ denote the measure on $\{\pm1\}\times\R$ with density, $p(s,y;h_\alpha)$,
from  \eqref{eq:density-of-sy}. Finally, let $((s^i,y^i))_i$ be i.i.d. draws from $\nu_\alpha$.
The main theorem of this section is the following.

\begin{thm}\label{thm:conv-spin-laws}
Assume that for $H'$, the limiting overlap distribution satisfies $\zeta(q_{*})>0$. For
each $\alpha\in\N$, 
\[
\left((s^i_{\alpha,N},y^i_{\alpha,N})\right)_{i}\to((s^i,y^i))_i
\]
in distribution.
\end{thm}
Recall now that $(s^{i},y^{i})$  satisfies  \eqref{eq:tap-pure-state}. 
As a consequence, we have the
following corollary.
\begin{cor}\label{cor:Cavity-TAP-finite-N}
In the setting of Theorem \ref{thm:conv-spin-laws}, we have that 
\[
\left(\left\langle s\right\rangle^\intercal _{\alpha,N}-\tanh\left(\left\langle y\right\rangle^\intercal _{\alpha,N}-\left(\xi'(1)-\xi'(q_{*})\right)\gibbs{s}^\intercal_{\alpha,N}\right)\right)_{\alpha\in \mathbb N}\to0
\]
in distribution.
\end{cor}
The goal of this section is to prove these two results. 
We begin by proving that the overlap distribution for $\nu_{\alpha,N}$
has a simple limit. We then prove  Portmanteau
type theorems for $(s_{\alpha,N},y_{\alpha,N})$. These results
allow us to conclude that statistics of $(s_{\alpha,N},y_{\alpha N})$
are a continuous functionals of the overlap distribution of $\nu_{\alpha,N}$ (not $\nu^\intercal_{\alpha,N}$).
Since the latter converges, we  then conclude Theorem \ref{thm:conv-spin-laws}. The proof of Corollary \ref{cor:Cavity-TAP-finite-N} is then  immediate.

\subsection{Convergence of overlaps within a cluster}

We now prove that the $\nu_{\alpha,N}$ are replica symmetric.  Fix $\alpha\in \mathbb N$.
Let $(\sigma^{i})_{i=1}^{\infty}$ be drawn  from $\nu_{\alpha,N}^{\tensor\infty}$ and consider $R_{N}$ to be the doubly infinite overlap array defined by
\[
R_N=\left(R(\sigma^{i},\sigma^{j})\right).
\]
Finally, let $Q$ be the deterministic matrix which is doubly infinite, all $1$ on the diagonal
and $q_*$ on the off-diagonal. We then have the following
theorem. 
\begin{thm}
\label{thm:conv-ovlp} We have that
\[
R_{N}\convdist Q.
\]
\end{thm}
\begin{proof}
 By standard properties of product spaces, it suffices
to show that for any $k$, 
\begin{equation}\label{eq:conv-ovlp-pf}
\E\int_{X_{\alpha}^k}F(R^k_N)d\nu_{\alpha,N}
\to F(Q^k).
\end{equation}
Here $F$ is some smooth function on $[-1,1]^{k^2}$ and by $R^k_{N}$ and $Q^k$ are the overlap matrix for $k$ i.i.d. draws
from $\nu_{\alpha,N}$ and the first $k-$by$-k$ entries of $Q$ respectively. It suffices to work
on the event that $X_{\alpha}$ is non-empty. 
Since $F$ is smooth, observe that
it suffices to show that
\[
\E\int_{X_{\alpha}^{k}}\norm{R^k_N-Q^k}_{1}d\nu_{\alpha,N}^{\tensor k}=o_N(1).
\]
To this end, observe that  
\[
\int_{X_{\alpha}^{k}}\norm{R^k_N-Q^k}_{1}d\nu_{\alpha,N}^{\tensor k}=k\cdot(k-1)\frac{\int_{X_{\alpha}^{2}}\abs{R_{12}-q_*}d\mu_{\alpha,N}^{\tensor 2}}{\mu_N (X_\alpha)^2},
\]
where $R_{12}$ is the overlap of two replica from $\nu_{\alpha,N}$ and
the diagonal terms cancelled. This goes to zero in probability by \prettyref{thm:pure-state-AU} items 4 and 5. 
\end{proof}

\subsection{Continuity and Portmanteau-type results}

We now collect some continuity and Portmanteau type theorems which
will be useful in the following.
\begin{lem}
\label{lem:portmanteau} For each $\alpha$, the convergence 
\[
((s^i_{\alpha,N},y^i_{\alpha,N})) \stackrel{(d)}{\to}((s^i,y^i))
\]
holds if any only if for every $k$, $d:[k]\to\{0,1\}$,
and  family of continuous bounded functions $\{\phi_i\}$,
\begin{equation}\label{eq:port-eq}
\E \prod_{i\in[k]}(s^{i}_{\alpha,N})^{d(i)}\phi_i(y^i_{\alpha,N})\to\E \prod_{i\in[k]}(s^{i})^{d(i)}\phi_{i}(y^i).
\end{equation}
 Furthermore, it is necessary and sufficient to take $\phi$
of polynomial growth. 
\end{lem}
\noindent This result is a standard consequence of the fact that 
 $s_{\alpha,N}$ are $\{\pm 1\}$ valued and $\{y_{\alpha,N}\}$
have uniformly bounded sub-Gaussian tails (see \prettyref{lem:(Localization)}),
so we omit its proof. 

Finally we note the following continuity result which is a consequence
of \prettyref{thm:conv-ovlp}. In the following, we let $Y_t =W_{\xi'(t)}$,
where $W_t$ denotes a standard Brownian motion.
\begin{lem}
\label{lem:conv-of-cosh-moments} 
For any $k,\ell\geq 1$  and any family of continuous bounded functions $\{\phi_i\}_{i\in[\ell]}$, we have that
\begin{align}
\E\int_{X_{\alpha}^{k+\ell}}\prod_{i\in[\ell]}\phi_i(y(\sigma^i)) \prod_{j={\ell}+1}^{\ell+k}\cosh(y(\sigma^{j}))&d\nu^{\tensor \ell+k}_{\alpha, N}\nonumber \\
\to\E[\prod_{i\in[\ell]}\E\left(\phi_{i}(Y_{1})\vert Y_{q_{*}}\right) &\cdot \E\left(\cosh(Y_{1})\vert Y_{q_{*}}\right)^{k}]\label{eq:cosh-moments-conv}
\end{align}
\end{lem}
\begin{proof}

Observe that for $(\sigma^i)$ fixed, then 
\[
F((\sigma^i))=\E\prod_{i\in[\ell]}\phi_i(y(\sigma^i)) \prod_{j={\ell}+1}^{\ell+k}\cosh(y(\sigma^{j}))
\]
is a continuous, bounded function of the overlap array $R$.
In particular, we may view it as a function of the form $F=F(\xi'(R)+o_N(1))$,
where by $\xi'(R)+o_N(1)$, we mean that we apply a function $f$ to $R$ coordinate
wise that satisfies the estimate $f=\xi' +o_N(1)$. 

Now, recall from \prettyref{eq:y-defn-indept}, that $y$ is independent of $H'$ by construction.
Thus it is independent of $\nu_{\alpha,N}$ and $X_\alpha$.
We may then integrate the lefthand side
of \prettyref{eq:cosh-moments-conv} first in $y$, to obtain
\[
\E\int_{X_{\alpha}^{k+\ell}}F(\xi'(R)+o_{N}(1))d\nu_{\alpha,N}^{\tensor k +\ell}.
\]
By a mollification argument, it suffices to study
the convergence of 
\[
\E\int_{X_{\alpha}^{k+\ell}}F(\xi'(R))d\nu_{\alpha,N}^{\tensor k+\ell}
\]
where this is the same function $F$ as above. By \prettyref{thm:conv-ovlp},
 this converges to $F(\xi'(Q))$. It remains to understand $F(\xi'(Q))$.
By the definition of the matrix $Q$, 
\[
F(\xi'(Q))=\E\left[\left(\prod_{i\in[\ell]}\E\left(\phi_{i}(Y_{1})\vert Y_{q_{*}}\right)\right)\E\left(\cosh(Y_{1})\vert Y_{q_{*}}\right)^{k}\right],
\]
as desired.
\end{proof}

\subsection{Proofs of main theorems}
We can now turn to the proofs of the main results.
 If $E$ is a measurable set and 
$f\in L^1(\mu)$ then we denote $\fint_E fd\mu=\frac{1}{\mu(E)}\int f d\mu $ with the convention that 
this is zero if $\mu(E)=0$.

\begin{proof}[\textbf{{Proof of \prettyref{thm:conv-spin-laws}}}] Fix $\alpha$. It suffices to work on the event  that $X_{\alpha}$ is non-empty. 
By \prettyref{lem:portmanteau}, it suffices to prove \prettyref{eq:port-eq}  for each
$n$, $d:[n]\to\{0,1\}$ and family of continuous bounded $\{\phi_i\}$.
 Furthermore, we claim that it suffices to prove
\begin{equation}
\E\prod_{i\in[n]}\gibbs{\phi_i(y)}^\intercal_{\alpha,N}\to\E\prod_{i\in[n]}\phi_i(y^i_{\alpha}).\label{eq:portmanteau2}
\end{equation}
To see this, simply note that
\begin{align*}
\E\prod (s^i_{\alpha,N})^{d(i)}\phi_i(y_{\alpha,N}) 
&= \E \prod \gibbs{s^{d(i)}\phi_i(y)}_{\alpha,N}^\intercal \\
& =\E\prod\frac{\int_{X_{\alpha}}\int_{\Sigma_1}\phi_i(y(\sigma))s^{d(i)}e^{sy(\sigma)}dsd\nu_N}{\int_{X_{\alpha}}2\cosh(y)d\nu_N}\\
 & =\E\prod\left\langle f_{d(i)}(y)\phi_{i}(y)\right\rangle^\intercal_{\alpha,N},
\end{align*}
where $f_{d}(x)=\tanh(x)$ if $d=1$ and $1$ if $d=0$.

With this claim in hand,  we now prove \eqref{eq:portmanteau2}.  To this end, fix
$\phi_{i}$ as above. By \eqref{eq:nu-alpha-T},
\begin{align*}
\E\prod_i\phi_i(y^i_{\alpha,N}) 
 & =\E\prod_i\frac{\fint_{X_{\alpha}}\phi_i(y(\sigma))\cosh(y(\sigma))d\nu_N}{\fint_{X_{\alpha}}\cosh(y(\sigma))d\nu_N}.
\end{align*}
Observe that $Z_{\alpha}=\fint_{X_{\alpha}}\cosh(y)d\nu_{N}$ satisfies $Z_{\alpha}\geq1$. By \prettyref{lem:(Localization)},
\[
P(Z_{\alpha}\geq L)\leq \frac{C({\xi})}{L}
\]
uniformly in $N$. Thus by a standard approximation argument, we can approximate $1/Z_{\alpha}^n$ by polynomials in $Z_{\alpha}$ in the above expectations. In particular, it suffices to study limits of integrals of the form
\[
\E \prod_i \fint_{X_\alpha} \phi_i(y(\sigma))\cosh(y(\sigma))d\nu_N\cdot(\fint_{X_\alpha}\cosh(y(\sigma))d\nu_N)^l.
\]
This is exactly of the form \prettyref{eq:cosh-moments-conv} with $k=l$, $\ell=1$ and the family $\{\phi_{i}(y) \cdot \cosh(y)\}_{i\in [n]}$ by Fubini's theorem. Thus
by \prettyref{lem:conv-of-cosh-moments}, 
\[
\E\prod_{i}\phi_{i}(y_{\alpha,N}^{i})\to\E\prod_i\frac{\E\left(\phi_i(Y_{1})\cosh(Y_{1})\vert Y_{q_{*}}\right)}{\E\left(\cosh(Y_{1})\vert Y_{q_{*}}\right)}.
\]

It remains to recognize the righthand side of the above display as an average with respect to $\nu_\alpha$. 
Observe that 
\begin{align*}
\frac{\E\left(\phi(Y_{1})\cosh(Y_{1})\vert Y_{q_{*}}\right)}{\E\left(\cosh(Y_{1})\vert Y_{q_{*}}\right)} 
& =\E\left(\phi(Y_{1})e^{\log\cosh(Y_{1})-\log\cosh(Y_{q*})-\frac{1}{2}(\xi'(1)-\xi'(q_{*}))}\vert Y_{q_{*}}\right)\\
 & \eqdist \int \phi(y)d\nu_\alpha,
\end{align*}
where 
the last equality is by definition. Thus
\begin{align*}
\E\prod\phi_i(y^i_{\alpha,N}) & \to\E\int\prod \phi_i(y^i)d\nu _{\alpha}^{\tensor n}
\end{align*}
as desired. \end{proof}

\begin{proof}[\textbf{{Proof of \prettyref{cor:Cavity-TAP-finite-N}}}]
Let $m_{\alpha,N}=\langle s\rangle_{\alpha,N}^\intercal$ and $h_{\alpha,N}=\langle y\rangle_{\alpha,N}^\intercal$. 
It suffices to show that 
for each $\alpha\in \mathbb N$, 
\[
(m_{\alpha,N},h_{\alpha,N})\convdist (\gibbs{s}_\alpha,\gibbs{y}_\alpha).
\]
Suppose first that this claim is true. Then the result immediately follows from 
\eqref{eq:tap-pure-state}.

We now turn to the claim. Observe that by \prettyref{lem:(Localization)}, these random variables
have sub-Gaussian tails. Thus it suffices to prove convergence of the moments
\[
\E m_{\alpha,N}^{k_1}h_{\alpha,N}^{k_2}.
\]
To this end, let $k=k_1+k_2$ and let $\{\psi_j\}_{j\in[k]}$ satisfy $\psi_j = 1$ if $i\leq k_1$ and $\psi_j(x)=x$ if $j>k_1$.
Finally let $d:[k]\to\{0,1\}$  be such that $d(i)=1$ if $i\leq k_1$ and $d(i)=0$ otherwise. Then,
by \prettyref{lem:portmanteau} and Theorem \ref{thm:conv-spin-laws}, we have that
\[
\E m_{\alpha,N}^{k_1}h_{\alpha,N}^{k_2} =\E \prod_j (s^j_{\alpha,N})^{d(j)}\psi_{j}(y^j_{\alpha,N})
\to\E \prod_j (s^j_{\alpha})^{d(j)}\psi_{j}(y^j_{\alpha})=\E \gibbs{s}_{\alpha}^{k_1}\gibbs{y}_{\alpha}^{k_2}
\]
as desired.  
\end{proof}

\section{Stability of clusters under lifts}\label{sec:stability}

In this section, we show that important properties of the pure states are carried over after lifting in one coordinate. We start with the following construction. 
For $\sigma = (\sigma_{1},\ldots, \sigma_{N}) \in \Sigma_{N}$, let $\rho(\sigma) = (\sigma_{2}, \ldots, \sigma_{N}) \in \Sigma_{N-1}$. For any mixed $p$-spin glass model, the Hamiltonian, $H_{N}$,
decomposes into a sum of three Gaussian processes:
\begin{equation}
H_{N}(\sigma)=\tilde H_{N}(\rho(\sigma))+\sigma_{1}y_{N}(\rho(\sigma))+r_{N}(\sigma_{1},\rho(\sigma)).\label{eq:ll}
\end{equation}
Properties of these Gaussian processes are described in Lemma \ref{lem:decomposition-lemma}. For $\sigma \in \Sigma_{N-1}$ set
$$H'_{N}(\sigma):= \tilde H_{N}(\sigma) + r_{N}(1,\sigma)$$
and let  $\mu_{N}'$ be the Gibbs measure corresponding to the Hamiltonian $ H_{N}'$. This Hamiltonian, and thus $\mu_{N}'$, is independent of $y_{N}$. We are thus in the setting of Section \ref{sec:cavity} where $H_{N}'$ satisfies \eqref{eq:ham-for-spin-section} and $y_{N}$ satisfies \eqref{eq:y-defn-indept}. 

Let $\tilde W_{\alpha,N-1}$, $\alpha \in \mathbb N$ be the subsets of $\Sigma_{N-1}$ constructed via Theorem \ref{thm:pure-state-AU}  relative to the measure $\mu_{N}'$. Set 
\begin{equation}\label{def:Wnda}
W_{\alpha,N}^{\dagger}= \Sigma_{1} \times \tilde W_{\alpha,N-1} \subset \Sigma_{N}.
\end{equation}
Order the sets $W_{\alpha,N}^{\dagger}$ with respect to their $\mu_{N}$ masses. That is, define subsets $W_{\alpha,N} \subset \Sigma_{N}$, for $\alpha \in \mathbb N$, such that
\begin{equation}\label{def:Wna}
\mu_{N}(W_{1,N}) \geq \mu_{N}(W_{2,N})\geq \ldots
\end{equation}
and so that
 $$W_{\alpha,N}=W_{\pi_N(\alpha),N}^{\dagger},$$ 
for some (random) automorphism $\pi_N:\mathbb N \to \mathbb N$. 
 
 \begin{rem}
Note that there is not a unique way to define the projection $\pi_N$ since,
there are possibly ties $W_\alpha =W_\beta$. Note, however, this only introduces
a finite indeterminacy as there are only finitely many such sets that are non-empty by 
construction. 
The reader can take any tie breaking rule. 
\end{rem}

The goal of this section is
to show that the collection $(W_{\alpha,N})_{\alpha \in \mathbb N}$  also satisfies 
items 1.-5. from Theorem \ref{thm:pure-state-AU}.
(For the rest of section, we drop the subscript $N$ of our notation.) 
The main idea is that at the level of overlaps, the measure $\mu$ on the sets
$W_\alpha$ will essentially be the same as the measure $(\mu')^\intercal$
on the sets $\tilde{W}_{\pi(\alpha)}$. Since on $\Sigma_{N-1}$,  $(\mu')^\intercal\gg\mu'$, 
overlap events that are rare for $\mu'$ will still be rare for $(\mu')^\intercal$. 
We begin by recording the following lemma which is a quantification of this observation. 

Recall the local field $y=y_{N}$ from \eqref{spinlocalfield} and the function $T$ from \eqref{eq:rad-nyk}. 
Let 
\begin{equation}\label{eq:K-tilde}
\tilde{K}(\mu')=\left(\int\cosh(2y)d\mu'\right)^{1/2}.
\end{equation}

\begin{lem}[Tilting Lemma] \label{lem:TL} There are constants $C,c>0$
such that with probability at least $1-\frac{1}{c}e^{-cN}$, \[
\left(1-\frac{C}{\sqrt{N}}\right)\int_{A}Td\mu'\leq\mu(\Sigma_{1}\times A)\leq\left(1+\frac{C}{\sqrt{N}}\right)\int_{A}T\mu', \quad \forall A \subset \Sigma_{N-1}.
\]
In particular, 
\[
\mu(\Sigma_{1}\times A)\leq\tilde{K}(\mu')\left(1+\frac{C}{\sqrt{N}}\right)\sqrt{\mu'\left(A\right)}.
\]
\end{lem}
\begin{proof}
This result immediately follows from Lemma \ref{lem:decomposition-lemma}. Observe
that if we let 
\[
\Delta=2\max_{\sigma \in \Sigma_{N-1}}\abs{r(1,\sigma)-r(-1,\sigma)},
\]
then 
\begin{equation*}
\mu(\Sigma_{1}\times A)  =\frac{\int_{A}\int_{\Sigma_{1}}e^{ \tilde H(\sigma)+\epsilon y(\sigma)+r(\epsilon,\sigma)}d\epsilon d\sigma}{\int_{\Sigma_{N-1}}\int_{\Sigma_{1}}e^{ \tilde H(\sigma)+\epsilon y(\sigma)+r(\epsilon,\sigma)}d\epsilon d\sigma}
  \leq \int_{A}T(\sigma)d\mu'e^{\Delta}.
\end{equation*}
Similarly 
\[
\mu(\Sigma_{1}\times A)\geq\int_{A}T(\sigma)d\mu'e^{-\Delta}.
\]
The first result then follows by Lemma \ref{lem:decomposition-lemma}, and the second result
follows from the first and the Cauchy-Schwarz inequality.
\end{proof}
We now start by proving the properties mentioned above.
\begin{lem}\label{lem:Wnaregood}
Let $q'_N=q_{N-1}, a'_{N}=a_{N-1},b'_{N}=b_N^{1/4}$, and $\epsilon_{N}'=\epsilon_N^{1/4}$. Then the sets
$\{W_{\alpha}\}_{\alpha\in[m_{N}]}$ satisfy items $1-4$
 Theorem \ref{thm:pure-state-AU}  with probability $1-o_{N}(1)$, where the sequences $q_N',a_N',b_N',\epsilon_N'$ and $m_N$
 satisfy those conditions.
\end{lem}
\begin{proof}
Since the sets $\tilde W_{\alpha}$ are disjoint,
 $W_{\alpha}^{\dagger}$ and $W_{\alpha}$ are as well and satisfy 
\[
(\cup_{\alpha} W_{\alpha})^{c} = (\cup_{\alpha} W_{\alpha}^{\dagger})^{c}=\left(\cup_{\alpha}\Sigma_{1}\times \tilde W_{\alpha}\right)^{c}=\Sigma_{1} \times \left(\cup_{\alpha} \tilde W_{\alpha}\right)^{c}.
\]
Thus by the Tilting Lemma (Lemma \ref{lem:TL}) and item 1 of Theorem \ref{thm:pure-state-AU}, we have
that with high probability, 
\begin{equation}\label{eq:dustWn}
\mu \left((\cup_{\alpha} W_{\alpha})^{c}\right)  \leq\left(1+\frac{C}{\sqrt{N}}\right)\tilde{K}(\mu')\cdot\sqrt{\epsilon_{N}}.
\end{equation}
Furthermore, by the Tilting Lemma  and item 2 of Theorem \ref{thm:pure-state-AU} ,
we obtain for $\beta =\pi^{-1}(\alpha)$
\begin{align*}
\mu^{\tensor2}\left(\sigma^{1},\sigma^{2}\in W_{\alpha}:\right.&\left.R_{12}\leq q_{N-1}-2a_{N-1}\right) \\
& \leq\tilde{K}(\mu')^{2}(1+\frac{C}{\sqrt{N}})\sqrt{\left(\mu'\right)^{\tensor2}\left(\sigma^{1},\sigma^{2}\in \tilde W_{\beta}:R_{12}\leq q_{N-1}-2a_{N-1}+\frac{1}{N}\right)}\\
 & \leq\tilde{K}(\mu')^{2}(1+\frac{C}{\sqrt{N}})\sqrt{b_{N}},
\end{align*}
where we used the fact that we may take $a_{N-1}\geq\frac{1}{N}$.
Argue similarly to get that for $\alpha_{1}\neq\alpha_{2}$, 
\[
\mu^{\tensor2}\left(\sigma^{1}\in W_{\alpha_{1}},\sigma^{2}\in W_{\alpha_{2}}:R_{12}\geq q_{N-1}+2a_{N-1}\right)\leq\tilde{K}(\mu')^{2}\left(1+\frac{C}{\sqrt{N}}\right)\sqrt{b_{N}}.
\]
Observe that by \prettyref{lem:(Localization)}, with probability tending to 1, $\tilde{K}(\mu')\leq b_N^{-\gamma}\vee \epsilon_N^{-1/4}$. This
yields the desired result after observing that since $\zeta_N[q_N+a_N,1]\geq \zeta\{q_*\} -b_N$, for $N$ sufficiently large, 
the same is true for $q'_N, a'_N$ and $b'_N$, and that item $4$ in \prettyref{thm:pure-state-AU} is implied by this fact
regarding $\zeta_N$ and items 2 and 3.
\end{proof}

It remains to show that the weights $\mu(W_{\alpha})$ converge to
a Poisson-Dirichlet process.
\begin{lem}\label{lem:convergencetoPD}
We have that 
\[
(\mu(W_{\alpha}))_{\alpha \in \mathbb N}\to(v_{\alpha})_{\alpha \in \mathbb N}
\]
in distribution on the space of mass partitions $\cP_{m}$.
\end{lem}
\begin{proof}
Recall that $\{\mu_N\}$ satisfy the approximate Ghirlanda-Guerra identities
since $H_{N}$ is a generic model. Let $U_{12}=U(\sigma^{1},\sigma^{2})$
be 
\[
U_{12}=\indicator{\exists\alpha\in\N:\sigma^{1},\sigma^{2}\in W_{\alpha}}
\]
and let $L_{N}=\left\{ \sigma^{1},\sigma^{2}\in\cup_{\alpha} W_{\alpha}\right\} $.
Then by the arguments of \cite[Section 6]{Jag14}, in order to prove that this sequence
converges, it suffices to prove that for some $\phi_{\kappa,\lambda}$
which satisfies 
\[
\phi_{\kappa,\lambda}(x)=\begin{cases}
0 & x\leq q_{*}-\kappa\\
1 & x\geq q_{*}-\lambda,
\end{cases}
\]
and interpolates between the two values for $x\in[q_{*}-\kappa,q_{*}-\lambda]$,
we have 
\[
\lim_{\kappa,\lambda\to0}\limsup_{N\to\infty}\E\left\langle \abs{U_{12}-\phi_{\kappa,\lambda}}\right\rangle =0.
\]
To see this, if we denote $\abs{U_{12}-\phi_{\kappa,\lambda}}=A$,
then 
\[
\E\left\langle A\right\rangle _{\mu}\leq\E\left\langle A\indicator{L_{N}}\right\rangle +o_{N}(1)
\]
where the fact that the second term is $o_{N}(1)$ follows from \eqref{eq:dustWn}.
Now 
\begin{align*}
\E\left\langle A L_{N}\right\rangle  & =\E\left\langle A\indicator{L_{N},R_{12}\ge q_{*}-\lambda}U_{12}\right\rangle +\E\left\langle A\indicator{L_{N},R_{12}\leq q_{*}-\lambda}U_{12}\right\rangle \\
 & \qquad+\E\left\langle A \indicator{L_{N},R_{12}\ge q_{*}-\kappa}(1-U_{12})\right\rangle +\E\left\langle A\indicator{L_{N},R_{12}\leq q_{*}-\kappa}(1-U_{12})\right\rangle \\
 & =I+II+III+IV.
\end{align*}
Note that $I=IV=0$ identically. It remains to estimate $II$ and
$III$.  

We start with $II$. Observe that
\[
II\leq2\E\left\langle U_{12}\left(\indicator{R_{12}\leq q_{*}-2a_{N-1}}\right)\right\rangle
\]
for $N$ large enough, which is bounded by $b_{N}'$ by  \prettyref{lem:Wnaregood}. 

Now to estimate $III$. Note that for $N$ sufficiently large, 
\[
\E\left\langle A(1-U_{12})\left(\indicator{R_{12}\geq q_{N-1}+2a_{N-1}}+\indicator{R_{12}\in[q_{*}-\kappa,q_{N}+2a_{N-1})}\right)\right\rangle _{\mu}\leq b_{N}'+(a).
\]
By the tilting lemma, 
\[
(a)\leq\norm{\tilde{K}}_{4}^{2}\cdot\left(\E\mu^{\tensor2}[q_{*}-2\kappa,q_{N-1}+a_{N-1})\right)^{1/2}.
\]
By the choice of $q_{N}$ and $a_{N}$ (see the first display in Theorem \ref{thm:pure-state-AU}), we have that
\[
\limsup\zeta_{N}[q_{*}-2\kappa,q_{N-1}+a_{N-1})=\limsup(\zeta_{N}[q_{*}-2\kappa,1]-\zeta_{N}[q_{N-1}+a_{N-1},1]) = 0.
\]
Thus combining these estimates and \prettyref{lem:(Localization)} we see that sending $N\to\infty$,
$\lambda\to0$ and then $\kappa\to0$ yields the result. 
\end{proof}

\section{Essential  uniqueness of clusters}\label{sec:Nesting}

In this section, we show that sets that satisfy the properties from \prettyref{thm:pure-state-AU}
with respect to $\mu$ are asymptotically unique. 

Let $\left\{ C_\alpha\right\} $ be constructed as in Theorem \ref{thm:pure-state-AU} for the measure $\mu_{N}$. Recall that they are labelled
in decreasing order, i.e., 
\[
\mu_{N}(C_\alpha)\geq\mu_{N}\left(C_{\alpha+1}\right).
\]
Let $a_{N},$ $b_{N}$, $m_N$, $q_{N}\to q_{*}$, and $\epsilon_{N}$ be
as in that theorem. Let $\{X_{\alpha}\}_{\alpha\in[m_N]}$ be another
collection of sets that satisfies items 1-5 of \prettyref{thm:pure-state-AU},
with constants $q_N',a_N',b_N'$ and $\epsilon_N'$ as in that theorem.

The main goal of this section is to prove that, the pure states $C_\alpha$ and
the sets $X_\alpha$ are effectively the same, as far as $\mu$ is concerned.
\begin{thm}[Essential uniqueness]\label{thm:nesting}
Suppose that we have
\begin{equation}\label{eq:overlapqn}
\zeta_{N}\left[(q_{N}'-a_{N}'),(q_{N}+a_{N})\right]+\zeta_{N}\left[(q_N-a_N),(q'_{N}+a'_{N})\right]\to0.
\end{equation}
 Then, for each $\alpha\in\N$, we have that 
\begin{equation}
\mu_{N}\left(C_\alpha\Delta X_\alpha\right)\to0\label{eq:nesting}
\end{equation}
in probability, where $\Delta$ denotes the symmetric difference.
\end{thm}

As a corollary of this we get the following.
\begin{cor} Let $W_\alpha$ be as in \prettyref{lem:Wnaregood}. Then \prettyref{eq:overlapqn} 
holds.
In particular,
\begin{equation}
\mu_{N}\left(C_\alpha\Delta W_\alpha\right)\to0\label{eq:nesting-1}
\end{equation}
\label{cor:quasi}
in probability.
\end{cor}
\begin{proof} 
This follows by Lemma \ref{lem:Wnaregood} and Theorem \ref{thm:nesting} after recalling that 
\begin{align*}
\zeta_N[q_N'+a_N',1] &\geq \zeta[q_*] -o_N(1)\\
\zeta_N[q_N+a_N,1] &\geq \zeta[q_*] -o_N(1).
\end{align*} 
Indeed, this implies that
$$ \zeta_N[q_N'-a_N',q_N+a_N] = \zeta_{N}[q_{N-1} - a_{N-1}, q_{N} + a_{N}] \to 0. $$
The same argument holds 
for the second limit.
\end{proof}

The idea of the proof \prettyref{thm:nesting} is that the overlap properties of the sets $(X_\alpha)$ and $(C_\alpha)$
from  items 1-4 of \prettyref{thm:pure-state-AU} will imply that each of the first $n$ $(X_{\alpha})$'s will be supported by one the first $M$ 
$(C_\alpha)$'s for some $M$ large but fixed, and \emph{vice versa}. The ranking of the states and basic properties
of the Poisson-Dirichlet process will then imply  
that, in fact, for each $\alpha$, the sets $X_\alpha$ and $C_\alpha$ are actually supported by each other.

For this we will need the following three lemmas. Their proofs are deferred to  the end of this section and follow from properties of the Poisson-Dirichlet process. The first lemma says that there is not much mass in the the tail  of the collections $X_{\alpha}$ and $C_{\alpha}$.  

\begin{lem}
\label{lem:no-tails-clusters}For every $\epsilon>0$, there is an
$N_{0}(\epsilon)$ and $M(\epsilon)$ such that if 
\[
E_{N}(\epsilon)=\left\{ \mu_{N}\left(\cup_{\alpha\geq M(\epsilon)}X_\alpha\right)>\frac{\epsilon}{2}\right\} \cup\left\{ \mu_{N}\left(\cup_{\alpha\geq M(\epsilon)}C_\alpha\right)>\frac{\epsilon}{2}\right\} 
\]
then for $N\geq N_{0}(\epsilon)$, 
\[
\prob\left[E_{N}(\epsilon)\right]\leq\epsilon.
\]
\end{lem}
The second lemma says that, for any fixed $n$, the first $n$ states $(C_{k})$ and $(X_k)$ must have
non-negligible $\mu_{N}$ mass as $N$ goes to infinity.

\begin{lem}
\label{lem:non-empty-first-n}Fix $n\geq1$ and $\delta>0$. Let $F_{N}(n,\delta)$
be the event that 
\begin{align*}
\mu_{N}(X_{1}) & >\ldots>\mu_{N}(X_n)>\delta\\
\mu_{N}(C_{1}) & >\ldots>\mu_{N}(C_{n})>\delta,
\end{align*}
then there is a function $f_{1,n}$ satisfying $\lim_{\delta \to0}f_{1,n}(\delta)=0$
and an $N_{1}(n,\delta)$ such that for $N\geq N_{1}(n,\delta)$,
\[
\prob\left[F_{N}(n,\delta)\right]\geq1-f_{1}(\delta).
\]
\end{lem}

The last lemma concerns the gap between the masses of states.

\begin{lem}
\label{lem:non-empty-gap} Fix $\eta>0$ and $n\geq1$. Let 
\begin{align*}
I_{N}(\eta,n) & =\left\{ \mu_{N}(C_{i})-\mu_{N}(C_{i+1})>\eta\:\forall i\in[n-1]\right\} \\
 & \qquad\cap\left\{ \mu_{N}(X_{i})-\mu_{N}(X_{i+1})>\eta\:\forall i\in[n-1]\right\} .
\end{align*}
Then there is a function $f_{2}(\eta,n)$ and an $N_{2}(\eta,n)$, such
that for $N\geq N_{2}(\eta,n)$, 
\[
\prob\left(I_{N}(\eta,n)\right)\geq1-f_{2}(\eta,n),
\]
where for each $n$, $f_{2}(\eta,n)\to0$ as $\eta\to0$. 
\end{lem}
Given $\varepsilon>0$, choose $\delta$, $\epsilon$,
and $\eta$ by combining \prettyref{lem:no-tails-clusters}-\ref{lem:non-empty-gap},
such that if 
\[
\cE_{N}(\epsilon,\delta,n,\eta):=E_{N}^{c}(\epsilon)\cap F_{N}(n,\delta)\cap I_{N}(\eta,n)\cap J_{N},
\] where
$J_{N}$ is the event that the conclusions of Theorem \ref{thm:pure-state-AU} hold
then 
\begin{equation}
\prob\left[\cE_{N} \right]>1-\varepsilon,\label{eq:var-epsilon}
\end{equation}
for all $N\geq N_{0}(\varepsilon)$.
\begin{proof}[\textbf{{Proof of \prettyref{thm:nesting}}}]
We want to show that for each $\rho>0,\varepsilon>0$ and $\alpha$, 
\begin{equation}
\prob\left(\mu_{N}\left(C_\alpha\Delta X_\alpha\right)>\rho\right)\leq\varepsilon.\label{eq:failure-to-nest}
\end{equation}
Fix $\rho,\varepsilon,$ and $\alpha$. Let $n>\alpha$. Let $N\geq N_{0}(\varepsilon/2)$ where $N_{0}$
is defined as in \eqref{eq:var-epsilon}. By \prettyref{eq:overlapqn} and Markov's inequality,
there is a $c_N\to0$ such that with probability $1-o_N(1)$,
\begin{equation}\label{eq:cn-def}
\mu^{\tensor 2}_N( R_{12}\in [q_N'-a_N',q_N+a_N] ) \leq c_N.
\end{equation}
Choose $N$ sufficiently
large that 
\[
\frac{2 M(\epsilon)}{\epsilon}( m_N(b_N+b_N')+c_N) +\epsilon_N<\frac{\rho\wedge\eta\wedge\epsilon}{2}
\]
where $\epsilon ,\eta$ are defined as above 
We can do this since by assumption,
\[
(b_{N}'+b_{N})\cdot m_{N}=o_{N}(1).
\]
We will prove shortly that on ${\cE}_{N}$, for 
\[
\iota_N = \frac{2 M(\epsilon)}{\epsilon}( m_N(b_N+b_N')+c_N) +\epsilon_N,
\]
we have that
\begin{align}
\mu_{N}\left(C_{\alpha}\backslash X_\alpha\right) & \leq\iota_N \nonumber, \\
\mu_{N}(X_{\alpha}\backslash C_{\alpha}) & \leq\iota_N\label{eq:nest-inequality}.
\end{align}
Note that \eqref{eq:nest-inequality} immediately implies \eqref{eq:failure-to-nest} as desired.
\end{proof}

\begin{proof}[Proof of  \eqref{eq:nest-inequality}]
 We begin by defining two maps $\pi_{1},\pi_{2}:[n]\to[M(\epsilon)]$.
On the event $\cE_{N}$, for each $i$, we let $\pi_{1}(i)$
be the first $j\in[M(\epsilon)]$ such that
\[
\mu_{N}\left(X_{i}\cap C_{\pi_{1}(i)}\right)\geq\frac{\epsilon}{2\cdot M(\epsilon)}
\]
holds and let $\pi_{2}(i)$ be the first $j\in[M(\epsilon)]$
such that 
\[
\mu_{N}\left(X_{\pi_{2}(i)}\cap C_{i}\right)\geq\frac{\epsilon}{2\cdot M(\epsilon)}
\]
holds. That such $j$ exist follows by definition of $\cE_{N}$.
On $\cE_{N}^{c}$, let $\pi_{1}=\pi_{2}=Id$. This provides two random
maps $\pi_{i}:[n]\to[M(\epsilon)]$, $i=1,2.$

Suppose for the moment that on ${\mathcal{E}}_{N}$,
\begin{align}
\mu_{N}\left(X_{i}\cap C_{\pi_{1}(i)}\right) & \geq\mu_{N}\left(X_{i}\right)-\iota_N \nonumber \\
\mu_{N}\left(C_{i}\cap X_{\pi_{2}(i)}\right) & \geq\mu_{N}\left(C_{i}\right)-\iota_N.\label{eq:strong-nesting-im-running-out-of-labels}
\end{align}
The inequality, \prettyref{eq:nest-inequality},
provided 
that $\pi_{1}=\pi_{2}=Id$ on $\mathcal{{E}}_{N}$.
Let us first show that these maps are the identity map given
\prettyref{eq:strong-nesting-im-running-out-of-labels}. We then prove \prettyref{eq:strong-nesting-im-running-out-of-labels}.

The proof that these maps are the identity map is by induction. Suppose first that $\pi_{2}(1)=1.$
If $\pi_{1}(1)>1$, then by \prettyref{eq:strong-nesting-im-running-out-of-labels},
\begin{align*}
\mu_{N}(C_{1}) & \leq\mu_{N}(X_{1})+\iota_{N}\\
 & \leq\mu_{N}\left(C_{\pi_{1}(1)}\right)+2\iota_{N}\leq\mu_{N}\left(C_{2}\right)+2\iota_{N}.
\end{align*}
This implies that
\[
\mu_{N}\left(C_{1}\right)-\mu_{N}\left(C_{2}\right)\leq2\iota_{N}.
\]
Since $\iota_{N}\to0$, this contradicts the definition of $\cE_{N}$. 
By symmetry, the same argument works if $\pi_1(1)=1$ and $\pi_2(1)>1$.

Now assume that $\pi_{2}(1)>1$ and $\pi_{1}(1)>1$. By the ordering
of these sets, 
\begin{align*}
\mu_{N}\left(C_{1}\right) & \leq\mu_{N}\left(X_{\pi_{2}(1)}\right)+\iota_{N}\leq\mu_{N}\left(X_{1}\right)+\iota_{N}\\
 & \leq\mu_{N}\left(C_{\pi_{1}(1)}\right)+2\iota_{N}.
\end{align*}
This is, again, a contradiction. Thus $\pi_{1}(1)=1=\pi_{2}(1).$

Assume now that $\pi_{1}(i)=\pi_{2}(i)=i$ for all $i\in[k-1]$. By
the same reasoning as in the base case, if $\pi_{2}(k)\neq k$, then
it must be that $\pi_{2}(k)<k$. This, however, implies that
\[
\mu_{N}\left(C_{k}\right)\leq\mu_{N}\left(X_{\pi_{2}(k)}\backslash C_{\pi_{2}(k)}\right)+\iota_N.
\]
But 
\begin{align*}
\mu_{N}\left(X_{\pi_{2}(k)}\backslash C_{\pi_{2}(k)}\right) & \leq\mu_{N}\left(X_{\pi_{2}(k)}\right)-\mu_{N}\left(X_{\pi_{2}(k)}\cap C_{\pi_{2}(k)}\right) \leq\iota_{N},
\end{align*}
where we used the induction hypothesis in the last inequality. This implies that eventually $\mu C_\alpha \leq 2\iota_N$. This
is, again, a contradiction since on $\cE_N$, $\mu C_\alpha>\epsilon$. Thus, assuming \prettyref{eq:strong-nesting-im-running-out-of-labels},
we have that $\pi_{1}=\pi_{2}=Id$ by induction.

We now prove \prettyref{eq:strong-nesting-im-running-out-of-labels}
on the event $\cE_{N}$. Fix $\alpha\in[n]$ . We know that
on this event,
\[
\mu_{N}\left(X_{{\alpha}}\cap C_{\pi_{1}(\alpha)}\right)\geq\frac{\epsilon}{2M(\epsilon)}.
\]
Now let $\ell\neq\pi_{1}(\alpha)$. Write
\begin{align*}
\mu_{N}\left(X_\alpha\cap C_{\ell}\right)
 & \color{red}=\color{black}\frac{1}{\mu_{N}\left(C_{\pi_{1}(\alpha)}\cap X_\alpha\right)}\mu^{\tensor2}_{N}\left(\sigma^{1}\in C_{\pi_{1}(\alpha)}\cap X_\alpha,\sigma^{2}\in C_{\ell}\cap X_\alpha\right).
\end{align*}
Write the event $\{R_{12}\in[-1,1]\}$ as 
\begin{align*}
\left\{ R_{12}\leq q_{N}'-a_{N}'\right\} \cup\left\{ R_{12}\geq q_{N}+a_{N}\right\}  & \cup\left\{ q_{N}'-a_{N}'<R_{12}<q_{N}+a_{N}\right\} \\
=I\cup II\cup III.
\end{align*}
Note that since we are in the event $J_{N}$,
\[
\mu^{\tensor2}_{N}\left(\sigma^{1}\in C_{\pi_{1}(\alpha)}\cap X_\alpha,\sigma^{2}\in C_{\ell}\cap X_\alpha,I\right)\leq\mu^{\tensor2}_{N}\left(\sigma^{1},\sigma^{2}\in X_\alpha,I\right)\leq b_{N}',
\]
while 
\[
\mu^{\tensor2}_{N}\left(\sigma^{1}\in C_{\pi_{1}(\alpha)}\cap X_\alpha,\sigma^{2}\in C_{\ell}\cap X_\alpha,II\right)\leq b_{N}.
\]

Summing on $\ell$ and using \eqref{eq:cn-def}, we see that 
\[
\sum_{\ell\neq \alpha}\mu_{N}\left(X_\alpha\cap C_{\ell}\right)\leq\frac{1}{\mu\left(C_{\pi_{1}(\alpha)}\cap X_\alpha\right)}\left(m_{N}(b'_{N}+b_{N})+c_N\right)\leq\frac{2M}{\epsilon}\cdot \left( m_{N}(b_{N}'+b_{N})+c_N\right).
\]
This implies the first inequality of \eqref{eq:strong-nesting-im-running-out-of-labels} after recalling that $\{C_{\ell}\}$ (almost) partitions
$\Sigma_{N}$ and that $$\mu_{N}\left(X_\alpha \cap (\cup_{\alpha} C_\alpha)^{c}\right) \leq \epsilon_{N},$$ by assumption.
By symmetry, the same argument shows the second inequality holds as well.
\end{proof}

\subsection{Propositions regarding the Poisson-Dirichlet process}

The proofs of Lemmas \ref{lem:no-tails-clusters}--\ref{lem:non-empty-gap}
follow by elementary applications of the Portmanteau lemma
combined with basic properties of the Poisson--Dirichlet process. 
For the reader's convenience we prove \prettyref{lem:no-tails-clusters}. 
The proofs of \prettyref{lem:non-empty-first-n} and \prettyref{lem:non-empty-gap} are omitted.

\begin{proof}[\textbf{{Proof of \prettyref{lem:no-tails-clusters}}}]

Fix $\epsilon>0$. Let $(v_{n})$ be $PD(1-\zeta(q_{*}))$. Let
$M(\epsilon)$ be such that 
\[
\prob\left(\sum_{\alpha\geq M(\epsilon)}v_{\alpha}\geq\frac{\epsilon}{2}\right)\leq \frac{\epsilon}{4}.
\]
Recall that $\left(\mu_{N}(X_{\alpha})\right)\to(v_{n})$ by Lemma \ref{lem:convergencetoPD}.
For $(v_\alpha^N)$, this event is contained in the closed event (in the topology of mass partitions) 
\[
\left\{ \sum_{\alpha\leq M} v_\alpha^N \leq 1-\epsilon/2\right\},
\]
and for $v_\alpha$ these events are equal. Thus we
have that for $N$ sufficiently large 
\[
\prob\left(\mu_{N}\left(\cup_{\alpha>M(\epsilon)}X_\alpha\right)\geq \frac{\epsilon}{2}\right)\leq \frac{\epsilon}{2},
\]
by the Portmanteau theorem. The same argument applies to the $C_\alpha$.
Intersecting these events yields the result by the inclusion-exclusion
principle.
\end{proof}

\section{TAP equation for a fixed coordinate}\label{sec:tap-fixed-coordinate}
In this section we turn to the proof of Theorem \ref{thm:TAP-finite-N}. For the reader's convenience, 
let us briefly recap where we are and our plan of attack. 
Recall the construction of the states $C_\alpha$ from Theorem \ref{thm:pure-state-AU} and the definition of $\langle \cdot \rangle_{\alpha}$. In Section \ref{sec:stability} we constructed another collection of pure states $W_\alpha \subseteq \Sigma_{N}$ for the measure $\mu_{N}$. 
As shown in Section \ref{sec:Nesting}, the sets  $C_\alpha$ and $W_\alpha$ are essentially the same 
in each other. The advantage of working with  $W_{\alpha}$ lies in the fact that they are rearrangements of lifts of pure states of the measure $\mu_{N-1}'$. This will allow us to avoid the first obstruction explained in \prettyref{sec:outline}: the measure $\mu_{N-1}'$ is now independent of the local field $y_{N}$. The rearrangement, however, is not independent of $y_{N}$. In particular the correlation between $W_\alpha$ and $y$ is through the map $\pi_N$ which takes $W^\dag_\alpha$ to $W^\dag_{\pi(\alpha)}=W_{\alpha}$.

To circumvent this obstruction we make the following observation. The measure $\mu$ conditioned on the set $W^\dag_\alpha$, is
essentially the measure $(\mu')^\intercal$ conditioned on $\tilde{W}_\alpha$. This  will allow us to conclude
that \eqref{eq:TAPmixed} holds by an application of Corollary \ref{cor:Cavity-TAP-finite-N}, provided the rearrangement map $\pi_N$ is not too wild.
In particular, provided the map $\mu'\mapsto(\mu')^\intercal$ does not
"charge the dust at infinity", the result will follow as a consequence of the following basic fact.

\begin{lem}\label{lem:rearrangement-of-sequences}
Let $X^{N}$ be a sequence of $[-2,2]^{\N}$-valued random variables
such that 
\[
X^{N}\stackrel{(d)}{\longrightarrow}0.
\]
 Let $p_{N}$ a sequence of $S_{\infty}-$valued random variables
that satisfy the tightness criterion
\begin{equation}\label{eq:tightness}
\varlimsup_{M\to\infty}\varlimsup_{N\to\infty}P(p_{N}(n)\geq M)=0\quad\forall n.
\end{equation}
Then if $Y^{N}=\left(X_{p_N(n)}^{N}\right)$, we have 
\[
Y^{N}\stackrel{(d)}{\longrightarrow}0.
\]
\end{lem}

We begin this section by proving the tightness of the sequence $\pi_N$. The main result will then essentially be immediate,
and is proved in the following subsection.

\subsection{Tightness of the reshuffling}

We begin this section by studying the random permutation $\pi:\N\to\N$ as 
defined in Section \ref{sec:stability} by
 $$W_{\alpha,N}=W_{\pi_N(\beta),N}^{\dagger}.$$ We recall its dependence on $N$ by writing $\pi_{N}$ instead of just $\pi$.
We now show tightness for the sequence $\pi_{N}$.
\begin{lem}[Tightness]\label{lem:tightness}
 We have that for each $n\in\N$, 
\[
\varlimsup_{M\to\infty}\varlimsup_{N\to\infty}P(\pi_N(n)\geq M)=0.
\]
\end{lem}
\begin{proof}
Take $N$ sufficiently large that $n\leq m_{N}$. Now observe that
\begin{align*}
P(\pi_N(n)\geq M) & =P\left(\exists k\geq M:\pi_{N}(k)=n\right)\\
 & =P(\exists l\leq n,k\geq M:\mu(W_{l}^{\dagger})\leq\mu(W_{k}^{\dagger}))\\
 & \leq\sum_{l=1}^{n}P(\exists k\geq M:\mu(W^\dagger_l)\leq\mu(W^\dagger_k))\\
 & \leq\sum_{l=1}^{n}P(\mu(W^\dagger_l)\leq\mu(\cup_{k\geq M}W^\dagger_k))
\end{align*}
It thus suffices to prove this limit for each summand.

Now observe that for each $l\in[n]$ and each $\epsilon>0$, the summand
satisfies the inequality,
\[
P(\mu(W^\dagger_l)\leq\mu(\cup_{k\geq M}W^\dagger_k)) \leq P(\mu(W^\dagger_l)\leq\epsilon)+P\left(\mu\left(\cup_{k\geq M}W^\dagger_k\right)\geq\epsilon\right)=I+II.
\]
We now bound $I$. Observe that by Lemma \ref{lem:TL}, 
\[
\mu\left(W^\dagger_l\right)\left\langle \cosh(y)\right\rangle '\geq(1-\frac{C}{\sqrt{N}})\mu'(\tilde{W}_{l})
\]
with high probability. Thus $I$ is bounded by 
\begin{align*}
I & \leq P(\mu'(\tilde{W}_{l})\leq2\epsilon\cdot L)+P(\left\langle \cosh(y)\right\rangle '\geq L)+o_{N}(1)\\
 & \leq P(\mu'(\tilde{W}_{l})\leq2\epsilon\cdot L)+\frac{C(\xi')}{L}+o_{N}(1),
\end{align*}
for each $L\geq1$, where we have applied the localization lemma (Lemma \ref{lem:(Localization)}) in
the second inequality. 

We now turn to $II$. Observe that again by Lemma \ref{lem:TL},
with high probability, 
\[
\mu\left(\cup_{k\geq M}W^\dagger_k\right)\leq(1+\frac{C}{\sqrt{N}})\tilde{K}(\mu')\sqrt{\mu'(\cup_{k\geq M}\tilde{W}_{k})}.
\]
Thus for $N$ sufficiently large, 
\begin{align*}
II & \leq P\left(2\tilde{K}(\mu')\sqrt{\mu'(\cup_{k\geq M}\tilde{W}_{k})}\geq\epsilon\right)+o_{N}(1)\\
 & \leq P(\mu'\left(\cup_{k\geq M}\tilde{W}_{k}\right)\geq\frac{\epsilon^{2}}{4L^{2}})+P\left(\tilde{K}(\mu')\geq L\right)+o_{N}(1)\\
 & =P(\mu'\left(\cup_{k\geq M}\tilde{W}_{k}\right)\geq\frac{\epsilon^{2}}{4L^{2}})+\frac{C(\xi')}{L}+o_{N}(1)
\end{align*}
where again in the last step we used  Lemma \ref{lem:(Localization)}. Denoting 
\[
\mu'(\tilde W_{k})=v_{k}^{N},
\]
we can write the above as 
\[
I+II\leq P(v_{l}^{N}\leq2\epsilon\cdot L)+P\left(\sum_{k\leq M}v_{k}^{N}\leq1-\frac{\epsilon^{2}}{4L^{2}}\right)+\frac{C}{L}+o_{N}(1).
\]
Observe that the sets in the first two terms are closed in $\cP_{m}$.
Thus by the Portmanteau theorem and the fact that $(v_{l}^{N})\to(v_{l})$
in law on $\cP_{m}$ where $(v_{l})$ are $PD(\theta)$ with $\theta=1-\zeta(\{q_{*}\})$,
we have that 
\[
\varlimsup_{N}I+II\leq P(v_{l}\leq2\epsilon\cdot L)+P\left(\sum_{k\geq M}v_{k}\geq\frac{\epsilon^{2}}{4L^{2}}\right)+\frac{C}{L}.
\]
We used here that for the Poisson-Dirichlet distribution $\sum v_k=1$. 

The Poisson-Dirichlet distribution satisfies  
\[
\E\sum_{k\geq M}v_{k}\leq f(M,\theta)
\]
where $f\to0$ as $M\to\infty$. In particular, by Markov's inequality
we have 
\[
P\left(\sum_{k\geq M}v_{k}\geq\frac{\epsilon^{2}}{4L^{2}}\right)\leq\frac{4L^{2}}{\epsilon^{2}}f(M,\theta).
\]
Thus combining the above we have that 
\[
\varlimsup_{M\to\infty}\varlimsup_{N\to\infty}P(\pi_N(n)\geq M)\leq nP(v_{n}\leq2\epsilon\cdot L)+n\frac{C}{L},
\]
where we have used here that $v_{n}<v_{k}$ for $k<n$. Sending $\epsilon\to0$
and then $L\to\infty$ and using the fact that $P(v_{n}=0)=0$, yields
the result. 
\end{proof}

\subsection{Proof of Theorem \ref{thm:TAP-finite-N} }
Recall the notation 
\[
\gibbs{\cdot}^\intercal_{\alpha,N}=(\mu')^\intercal_N\left(\cdot \vert \tilde{W}_\alpha\right)
\]
from \prettyref{sec:cavity} and recall that $\langle \cdot \rangle_{\alpha,N} = \mu_{N}( \cdot | C_{\alpha})$. 
We begin by stating the following two lemmas whose proofs we will
defer to the end of the section.
\begin{lem}
\label{lem:s-to-sbar}For every $\alpha\in\N$, 
\[
\abs{\left\langle \sigma_{1}\right\rangle _{\alpha,N}-{\left\langle \sigma_{1}\right\rangle }^\intercal_{\pi_N(\alpha),N}}\to0
\]
in probability as $N\to\infty.$
\end{lem}
\begin{lem}
\label{lem:tanh-to-tanhbar}For every $\alpha\in\N$, 
\begin{equation}
\abs{\tanh\left(\left\langle y\right\rangle _{\alpha,N}-(\xi'(1)-\xi'(q_{*}))\left\langle \sigma_{1}\right\rangle _{\alpha,N}\right)-\tanh\left({\left\langle y\right\rangle}^\intercal_{\pi_N(\alpha),N}-(\xi'(1)-\xi'(q_{*})){\left\langle \sigma_{1}\right\rangle}^{\intercal}_{\pi_N(\alpha),N}\right)}\to0\label{eq:tanh-difference}
\end{equation}
in probability as $N\to \infty$.
\end{lem}
\begin{proof}[\textbf{{Proof of Theorem \ref{thm:TAP-finite-N}}}]
 By the above two lemmas, it suffices to prove \eqref{eq:TAPmixed}
with ${\left\langle \cdot\right\rangle }^\intercal_{\pi_N(\alpha),N}$ replacing $\langle \cdot \rangle_{\alpha,N}$.  

Now, let $Y_{\alpha}^{N}={\left\langle \sigma_{1}\right\rangle }^\intercal_{\pi_N(\alpha),N}-\tanh\left({\left\langle y\right\rangle }^\intercal_{\pi_N(\alpha),N}-(\xi'(1)-\xi'(q_{*})){\left\langle \sigma_{1}\right\rangle }^\intercal_{\pi_N(\alpha),N}\right)$.
Note that $Y_{\alpha}^{N}$ can be written as $Y_{\alpha}^{N} = X_{\pi(\alpha)}^{N}$ where 
\[
X_{\alpha}^{N}:={\left\langle \sigma_{1}\right\rangle }^\intercal_{\alpha,N}-\tanh\left({\left\langle y\right\rangle }^\intercal_{\alpha,N}-(\xi'(1)-\xi'(q_{*})){\left\langle \sigma_{1}\right\rangle }^\intercal_{\alpha,N}\right),
\]
By Lemma \ref{lem:rearrangement-of-sequences} and Lemma \ref{lem:tightness}, it thus suffices to prove convergence of $X_{\alpha}^{N}$ to zero. 

Observe that for $X_\alpha^{N}$, this is a statement about a cavity coordinate with the local field independent of the measure $\mu'$. 
 Indeed, the Hamiltonian $H'$ satisfies \prettyref{eq:ham-for-spin-section}, and $y$ satisfies \eqref{eq:y-defn-indept}. Thus, $X_{\alpha}^{N}$ goes to zero in probability by Corollary \ref{cor:Cavity-TAP-finite-N} and Theorem \ref{thm:TAP-finite-N} follows.
\end{proof}

We now turn to the proofs of the lemmas.
Set 
\[
\tilde{\left\langle \cdot\right\rangle }_{\alpha,N}=\mu_{N}\left(\cdot\vert W_{\alpha}\right).
\]

\begin{proof}[\textbf{{Proof of \prettyref{lem:s-to-sbar}}}]
We begin by observing that 
\begin{align*}
\abs{\left\langle \sigma_{1}\right\rangle _{\alpha,N}-\tilde{\left\langle \sigma_{1}\right\rangle }_{\alpha,N}} & =\abs{\fint_{C_{\alpha}}\sigma_{1}d\mu_{N}-\fint_{W_{\alpha}}\sigma_{1}d\mu_{N}}\leq\frac{2\mu_{N}(W_{\alpha}\Delta C_{\alpha})}{\mu_{N} (C_{\alpha})}
\end{align*}
on the event that $W_{\alpha}$ and $C_{\alpha}$ both have positive
mass. Since $\mu_{N}\left(W_{\alpha}\Delta C_{\alpha}\right)\to0$
in probability by the essentially uniqueness theorem (Corollary \ref{cor:quasi}) and $\mu_{N}(C_{\alpha})$ and $\mu_{N}(W_{\alpha})$ converge
in law to a random variable that is almost surely positive, this goes to zero in probability.
Then note that by the tilting lemma, 
\[
\abs{\tilde{\gibbs{\sigma_{1}}}_{\alpha,N}-\gibbs{\sigma_{1}}_{\pi_N(\alpha)}^\intercal} \leq \frac{C'}{\sqrt{N}}
\]
with high probability, so that this too goes to zero in probability. The result then follows by the
triangle inequality.
\end{proof}

\begin{proof}[\textbf{{Proof of \prettyref{lem:tanh-to-tanhbar}}}]
 As $\tanh(x)$ is $1$-Lipschitz, and we know from \prettyref{lem:s-to-sbar}
that $\left\langle \sigma_{1}\right\rangle _{\alpha}-\tilde{\left\langle \sigma_{1}\right\rangle }^\intercal_{\pi_N(\alpha)}\to0$
in probability, it suffices to show that 
\[
\left\langle y\right\rangle _{\alpha,N}-\left\langle y\right\rangle^\intercal_{\pi_N(\alpha),N}\to0
\]
in probability. Observe that  
\[
\abs{\fint_{C_{\alpha}}yd\mu_N-\fint_{W_{\alpha}}yd\mu_N}\leq\frac{1}{\mu_{N} (C_{\alpha})}\norm{y}_{L^{2}(\mu)}\sqrt{\mu_{N}(W_{\alpha}\Delta C_{\alpha})}\left(1+\frac{1}{\sqrt{\mu_{N}(W_{\alpha})}}\right),
\]
and that with probability tending to 1, $(1+\frac{1}{\sqrt{\mu_{N}(W_{\alpha})}})\vee\frac{1}{\mu_{N}(C_{\alpha})}$
will be finite. Furthermore, $\mu_{N}(W_{\alpha}\Delta C_{\alpha})\to0$
in probability by the quasi-uniqueness theorem (Theorem \ref{thm:nesting}), and $\E\norm{y}_{2}\leq C$
uniformly in $N$ by item 3 of Lemma \ref{lem:(Localization)}. Thus this tends to zero in probability as before.
Similarly 
\[
\abs{\tilde{\gibbs{y}}_{\alpha,N}-\gibbs{y}_{\pi_N(\alpha)}^\intercal} \leq \frac{C'}{\sqrt{N}}\norm{y}_{L^2(\mu')}
\]
which goes to zero in probability by the same argument.
\end{proof}

\appendices
\section{Appendix\label{sec:Appendix}}

\subsection{The clusters $C_{\alpha,N}$ and approximate ultrametricity}\label{app:AU}

In this short section we summarize the properties of the clusters $C_{\alpha,N}$ used to construct the measures $\langle\cdot\rangle_{\alpha,N}$.
These properties are described in the following theorem, which is a rephrasing of the main results in \cite{Jag14}, specifically as in Section 9, Proposition 9.5-6 and Corollary 9.7 of that paper.

\begin{thm}[ \cite{Jag14}]\label{thm:pure-state-AU}
Assume that $\zeta(\{q_{*}\})>0$. Then there are sequences $q_{N} \uparrow q_{*}$, $\epsilon_{N}, a_{N}, b_{N}$ all converging monotonically to $0$, and $m_N\to\infty$, such that
$m_N\cdot b_N^{\gamma }\to0$ for some $\gamma\leq1$,  $q_{N} + a_{N} < q_{*}$ and 
\[
\zeta_{N}[q_{N}+a_{N}, 1] \geq \zeta(\{ q_{*}\}) -b_{N}
\]
for $N$ sufficiently large and such that with probability  $1-o_N(1)$, 
there exist disjoint random sets $\{C_{\alpha,N}\}_{\alpha \in \mathbb N}$ of $\Sigma_{N}$ :
\begin{enumerate}
\item The collection $C_{\alpha,N}$ exhaust the set $\Sigma_{N}$: $$\sum_{\alpha} \mu_{N}(C_{\alpha,N}) \geq 1 - \epsilon_{N}.$$
\item For any $\alpha$, points are uniformly close: $$\mu_{N}^{\tensor 2}\big(\sigma^{1},\sigma^{2} \in C_{\alpha,N}: R(\sigma^{1},\sigma^{2})\leq q_{N}-a_{N}\big)\leq b_{N}.$$
\item For any $\alpha \neq \beta$, $$\mu_{N}^{\tensor 2}\big(\sigma^{1}\in C_{\alpha,N},\sigma^{2} \in C_{\beta,N}: R(\sigma^{1},\sigma^{2})\geq q_{N}+a_{N}\big)\leq b_{N}.$$

\item Uniformly in $\alpha$ we have,
$$\int_{C_{\alpha,N}^{\tensor 2}} |R_{12}-q_{*}| d\mu_{N}^{2} < o_N(1).$$

\item The weights $(\mu_{N}(C_{\alpha,N}))$ are labeled in decreasing order of mass, and converge to the weights of a Poisson-Dirichlet process of parameter $1-\zeta(\{q_{*}\})$.

\end{enumerate}
\end{thm}
\noindent \textbf{Note:} We may always take $\alpha_{N-1}\geq N^{-1}$ in the above by monotonicity. That we can,
take $m_N\cdot b_N^{\gamma}\to0$, follows by adding a constant to the definition of $n_0$ in Lemma 5.2 of \cite{Jag14}.

\subsection{Tail bounds for some Gibbs averages}\label{app:localization}

\begin{lem}[Localization Lemma]
\label{lem:(Localization)} Recall $\tilde K(\mu')$ from \eqref{eq:K-tilde}, $y_{\alpha,N}$ from Section \ref{sec:cavity} and $y_{N}$ from \eqref{spinlocalfield}. For any $L >0$ we have the following estimates.

\begin{enumerate}
\item For any $\alpha \in \N$, 
\[
P(\abs{y_{\alpha,N}}>L)\leq C_{1}(\xi')\cdot e^{-C_{2}(\xi')L^{2}}.
\]
\item For any $\alpha \in \mathbb N$,
\[ P\left (\int_{X_{\alpha}} \cosh (y_{\alpha,N}) d\nu_{N} >L\right) \leq C(\xi')/L. \]

\item We have that 
\[P(\tilde{K}(\mu')\geq L)\leq C(\xi')/L.\]
\item We have that 
\[
\E\left(\int_{\Sigma_{N}}y^{2}_{N}d\mu_{N}\right)^{1/2}\leq C(\xi').
\]
\end{enumerate}
\end{lem}
\begin{proof} In the following we will drop the index $\alpha$ of our notation without any loss.
To see the first item, note that  $y_{\alpha,N}$ has finite moment generating
function. Fix $\lambda \geq 1.$ We have
\begin{align}\label{eq:daskioda}
\E e^{\lambda y_{N}} & =\E\left[\frac{\int_{X_{\alpha}}e^{\lambda y_{N}(\sigma)}\cosh(y_{N}(\sigma))d\nu_{N}}{\int_{X_{\alpha}}2\cosh(y_{N}(\sigma))d\nu_{N}}\right]+P(X_{\alpha}=\emptyset) \nonumber \\
 & \leq\E \left[\fint_{X_{\alpha}}e^{\lambda y_{N}(\sigma)}\cosh(y_{N}(\sigma))d\nu_{N} \right]+1 \nonumber\\
 & \leq\E\fint_{X_{\alpha}}\E(\exp(\lambda y_{N})\cosh(y_{N})\vert\nu_{N})d\nu_{N}+1 \nonumber\\
 & =\frac{1}{2}\left(e^{(1+\lambda)^{2}\xi'(1)}\right)+1.
\end{align}
Then, by Markov's inequality, we have 
\[
P(y_{N}\geq L)\leq \E e^{\lambda y_{N} - \lambda L} \leq  \E e^{(1+\lambda)^{2}\xi'(1) - \lambda L}\leq C_{1}(\xi')\cdot e^{-C_{2}(\xi')L^{2}},
\]
for $L$ sufficiently large by choosing $\lambda = L/2$, for instance. Increasing the value of $C_{1}(\xi')$ if necessary we obtain the result for all $L>0$. 
Similarly for $-y_{N}$.

The second item holds by Markov's inequality, conditioning on $\nu_{N}$ and using the Gaussian bound of item 1. For the third item, note that 
using Lemma \ref{lem:decomposition-lemma}, conditioning on $\mu'$ and letting $Z$ be a Gaussian random variable with variance $\xi'(1)$, we have
$$ P(\tilde K(\mu')\geq L)\leq C(\xi') \mathbb E \left[ e^{-4 Z} \cosh Z\right ] \mu'_{N}(\Sigma_{N}) \leq L^{-1}C(\xi'),$$
as desired.

We prove the last item as follows. To see this observe that it suffices to bound $\E\int y^{2}d\mu$.
To estimate this, observe that if $\Delta=\max\abs{r(1,\sigma)-r(-1,\sigma)}$,
then
\begin{align*}
\E\int y^{2}d\mu & \leq\E e^{2\Delta}\frac{\int_{\Sigma_{N-1}}y^{2}\cosh(y)d\mu'}{\int_{\Sigma_{N-1}}\cosh(y)d\mu'}\\
 & \leq\left(\E e^{4\Delta}\right)^{1/2}\left(\E\int y^{4}\cosh(y)d\mu'\right)^{1/2},\\
\end{align*}
where in the last inequality we use Cauchy-Schwarz and the fact that
$\cosh(x)\geq1$. Observe that the first term is bounded by \eqref{eq:Delta}.
Since $y$ is independent of $\mu'$ , we can integrate in $y$ to
find that the second term is also uniformly bounded.
 
\end{proof}

\subsection{Decomposition and regularity of mixed $p$-spin Hamiltonians}\label{app:decomposition}

In this section, we present some basic properties of mixed $p$-spin Hamiltonians.
Recall that for $\sigma = (\sigma_{1},\ldots, \sigma_{N}) \in \Sigma_{N}$, $\rho(\sigma) = (\sigma_{2}, \ldots, \sigma_{N}) \in \Sigma_{N-1}$. Now observe that for any mixed $p$-spin glass model, the Hamiltonian
has the following decomposition:
\begin{equation}
H_{N}(\sigma)=\tilde H_{N}(\rho(\sigma))+\sigma_{1}y_{N}(\rho(\sigma))+r_{N}(\sigma_{1},\rho(\sigma)),
\end{equation}
where the processes come from the following lemma.
\begin{lem}
\label{lem:decomposition-lemma} There exist centered Gaussian processes $\tilde H_{N},y_{N},r_{N}$
such that \eqref{eq:ll} holds and 
\begin{align*}
\mathbb{E}\tilde H_{N}(\sigma^{1})\tilde H_{N}(\sigma^{2})= & N\xi\left (\frac{{N-1}}{N}R_{12}\right),\\
\mathbb{E}y_{N}(\sigma^{1})y_{N}(\sigma^{2})= & \xi'(R_{12})+o_{N}(1),\\
\mathbb{E}r_{N}(\sigma^{1})r_{N}(\sigma^{2})= & O(N^{-1}).\\
\end{align*}
Furthermore, there exist positive constant $C_{1}$ and $C_{2}$ so
that with probability at least $1-e^{-C_{1}N},$
\[
\max_{\sigma\in\Sigma_{N-1}}|r_{N}(1,\sigma)-r_{N}(-1,\sigma)|\leq\frac{{C_{2}}}{\sqrt{N}},
\]
and a positive constant $C_{3}$ so that 
\begin{equation}\label{eq:Delta}
 \mathbb E \exp\bigg(2 \max_{\sigma \in \Sigma_{N-1}} |r_{N}(1,\sigma) - r_{N}(-1,\sigma)|\bigg) \leq C_{3}.
\end{equation}

\end{lem}
\begin{proof} The lemma is a standard computation on Gaussian processes. To simplify the exposition we will consider the pure $p$-spin model. The general case follows by linearity. Here, we set $$\tilde H_{N}(\rho(\sigma)) = N^{-\frac{p-1}{2}}  \sum_{2 \leq i_{1},\ldots, i_{p} \leq N} g_{i_{1}\ldots i_{p}} \sigma_{i_{1}}\ldots \sigma_{i_{p}},$$
$$y_{N}(\rho(\sigma)) = N^{-\frac{p-1}{2}}  \sum_{k=1}^{p}\sum_{\stackrel{2 \leq i_{1},\ldots, i_{p} \leq N}{i_{k}=1}} g_{i_{1}\ldots i_{p}}\sigma_{i_{1}}\ldots \sigma_{i_{p}}, $$
and 
$$r_{N}(\sigma_{1}, \rho(\sigma))=  N^{-\frac{p-1}{2}} \sum_{l=2}^{p} \sigma_{1}^{\ell} \sum_{2\leq i_{1}, \ldots, i_{p-\ell} \leq N} J_{i_{1}\ldots i_{p-\ell}} \sigma_{i_{1}} \ldots \sigma_{i_{p-\ell}},$$
where $J_{i_{1}\ldots i_{p-\ell}}$ are centered Gaussian random variables with variance equal to $\binom{p}{\ell}$: $J_{i_{1}\ldots i_{p-\ell}}$ is the sum of the $g_{i_{1} \ldots i_{p}}$ where the index $1$ appears exactly $\ell$ times.
Computing the variance of these three Gaussian processes give us the the first three statements of the Lemma. For the second to last  and last statement, note that for any $\sigma \in \Sigma_{N-1}$, $r(1,\sigma)-r(-1,\sigma)$ is a centered Gaussian process with variance equal to 
$$ \frac{4}{N^{p-1}} \sum_{\ell = 3, \:\ell \text{ odd}}^{p} \binom{p}{\ell} (N-1)^{p-\ell} \leq \frac{C_{p}}{N^{2}}, $$ 
for some constant $C_{p}$.
A standard application of Borell's inequality (Theorem 2.1.1 in \cite{ATbook}), the tail estimate for the maximum of a Gaussian process (Equation (2.1.4) in \cite{ATbook}) and Sudakov-Fernique inequality (Theorem 2.2.3 in \cite{ATbook}) gives us the desired result. 
\end{proof}

\bibliographystyle{plain}
\bibliography{localfields}
\end{document}